\documentclass[12pt]{article}
\usepackage{latexsym}
\usepackage[centertags]{amsmath}
\usepackage{amsfonts}
\usepackage{amssymb}

      \newtheorem{theorem}{Theorem}[section]

      \newtheorem{lemma}[theorem]{Lemma}

      \def\nn{\nonumber}
      \def\rf#1{\mbox{$(\ref{#1})$}}

      \def\be{\begin{equation}} 
      \def\ee{\end{equation}} 
      \def\beqn{\begin{eqnarray}} 
      \def\eeqn{\end{eqnarray}} 
      \def\beq{\begin{eqnarray*}} 
      \def\eeq{\end{eqnarray*}}
      \def\proof{{\bf Proof:}\ }
      \def\mb{\mbox} 
      \def\ga{\gamma} 
      \def\ep{\varepsilon} 
      \def\la{\lambda} 
      \def\ra{\rightarrow} 
      \def\lra{\longrightarrow} 

\newcommand{\eqdist}{\stackrel{\text{\upshape d}}{=}}

\def\ep{{\epsilon}}\def\ga{{\gamma}}

\def\<{\left<}\def\>{\right>}
\def\({\left(}\def\){\right)}

\newfam\msbmfam\font\tenmsbm=msbm10\textfont
\msbmfam=\tenmsbm\font\sevenmsbm=msbm7

\scriptfont\msbmfam=\sevenmsbm

\begin{document}

      \title{\bf  Limit Theorems Associated With The Pitman-Yor Process}
      \author{ Shui Feng \\McMaster
      University \and  Fuqing Gao \\ Wuhan University \and Youzhou Zhou\\ Zhongnan University of Economics and Law}
      \date{February 20, 2016}
      \maketitle
      \begin{abstract}
      
            The Pitman-Yor process is a random discrete measure. The random weights or masses follow the two-parameter Poisson-Dirichlet distribution with parameters $0<\alpha<1, \theta>-\alpha$. The parameters $\alpha$ and $\theta$ correspond to the stable and gamma components, respectively. The distribution of atoms is given by a probability $\nu$.  In this article we consider the limit theorems for the Pitman-Yor process and the two-parameter Poisson-Dirichlet distribution. These include law of large numbers, fluctuations, and moderate or large  deviation principles. The limiting procedures involve either $\alpha$ tends to zero or one.  They arise naturally in genetics and physics such as  the asymptotic coalescence time for explosive branching process and  the approximation to generalized random energy model for disordered system.

      \vspace*{.125in} \noindent {\it Keywords:}  Pitman-Yor process, Explosive branching, Large deviations, Phase transition, Poisson-Dirichlet distribution, Random energy model.
     
      \vspace*{.125in} \noindent {\it AMS 2001 subject classifications:}
      Primary 60F10; secondary 92D10.
      \end{abstract}
      \section{Introduction}
      
       For any $0\leq \alpha<1, \theta+\alpha>0$, let $U_1(\alpha,\theta),U_2(\alpha,\theta),\ldots$ be a sequence of independent random variables with $U_i(\alpha,\theta)$ having distribution $Beta(1-\alpha,\theta+i\alpha)$ for $i \geq 1$. If we define
      \[
      V_1(\alpha,\theta) =U_1(\alpha,\theta), V_n(\alpha,\theta)=(1-U_1(\alpha,\theta))\cdots(1-U_{n-1}(\alpha,\theta))U_n(\alpha,\theta), \ \ n\geq 2,
      \]
      then the law of the decreasing order statistic $${\bf P}(\alpha,\theta)=(P_1(\alpha,\theta),P_2(\alpha,\theta),\ldots)$$ of $(V_1(\alpha,\theta),V_2(\alpha,\theta),\ldots)$ is the two-parameter Poisson-Dirichlet distribution $PD(\alpha,\theta)$. It is  a probability on the infinite-dimensional simplex
      \[
       \nabla_{\infty}=\{{\bf p}=(p_1,p_2,\ldots): p_1\geq p_2\geq \cdots \geq 0, \sum_{i=1}^{\infty}p_i \leq 1\}.   
       \]
      Let $S$ be  Polish space and $\nu$ a probability on $S$ satisfying $\nu(\{x\})=0$ for all $x$ in $S$. In this case we say $\nu$ is diffuse. The Pitman-Yor process with parameters $\alpha,\theta$ and $\nu$ is the random measure
     \[
    \Xi_{\alpha,\theta,\nu}=\sum_{i=1}^{\infty}P_i(\alpha,\theta)\delta_{\xi_i}. 
     \]
  where $\xi_1,\xi_2,\ldots$ are i.i.d. with common distribution $\nu$ and is independent of  ${\bf P}(\alpha,\theta)$.  The case $\alpha=0$
 corresponds to the Dirichlet process constructed in \cite{Fer73}.
       
      The distribution $PD(0,\theta)$ was introduced by Kingman in  \cite{Kingman75} as the law of relative jump sizes of a gamma subordinator over the interval $[0,\theta]$. It also arises in other context most notably population genetics.  The distribution  $PD(\alpha,0)$  was introduced in Kingman \cite{Kingman75} through the stable subordinator.
  In \cite{PPY92} and \cite{PY92},  $PD(\alpha,0)$ was constructed from the ranked length  of excursion intervals between zeros of a Brownian motion ($\alpha=1/2$) or a recurrent Bessel process of order $2(1-\alpha)$ for general $\alpha$.     
  
    In this paper we focus on the case $\theta=0$. Without the loss of generality, we choose the space $S$ to be $[0,1]$ and the probability $\nu$ to be the uniform distribution on $[0,1]$. This implies that the parameter $\alpha$ is in $(0,1)$. Our main objective is to study the asymptotic behaviour of $PD(\alpha,0)$ when $\alpha$ converges to zero, and the behaviour of both $PD(\alpha,0)$ and $\Xi_{\alpha,0,\nu}$ when $\alpha$ converges to one.  There are many scenarios  where  the limiting procedure of $\alpha$ approaching one or zero arises naturally.  We consider two examples below.  
      
    The first example is  Derrida's random energy model (REM) introduced in \cite{Derrida80} and \cite{Derrida81}. This is a toy model for disordered system such as spin glasses. For any $N\geq 1$, let $S_N=\{-1,1\}^N$ denote the configuration space. Then the REM is  a family of i.i.d. random variables $\{H_N(\sigma): \sigma \in S_N\}$ with common normal distribution of mean zero and variance $N$. Here $H_N(\sigma)$ is the Hamiltonian. Given the temperature $T$ and $\beta=T^{-1}$,  the Gibbs measure is a probability on $S_N$ given by
      \[
      Z_N^{-1}{\exp\{-\beta H_n(\sigma)}\}
      \]
      where
      \[
      Z_N =\sum_{\sigma \in S_N}\exp\{-\beta H_N(\sigma)\}
      \]
   is the partition function.  Let $T_c = \frac{1}{\sqrt{2\ln 2}}$ and $\alpha =\frac{T}{T_c}$. Then for $T <T_c$ or equivalently $\beta >\sqrt{2\ln 2}$,  the decreasing order statistic of the Gibbs measure is known (cf. \cite{Tala03}) to converge to the Poisson-Dirichlet distribution $PD(\alpha,0)$ as $N$ tends to infinity.  Thus $\alpha$ converging to zero corresponds to temperature going to zero while $\alpha$ converging to one corresponds to temperature rising to the critical value.  To account for correlations, the generalized random energy model (GREM) involving hierarchical levels was introduced and studied in  \cite{Derrida85} and \cite{DerGar86}.  The generalization to continuum levels was done in  \cite{BoSz98} and the genealogy of the hierarchical systems is described by the Bolthausen-Sznitman coalescent.  In deriving the infinitesimal rate of the coalescent (Proposition 4.11 in \cite{Ber06} ), one needs to consider the limit of  $PD(e^{-t},0)$ as $t$ converges to zero or equivalently $\alpha =e^{-t}$ converging to one. 
   
The second example is concerned with the coalescence  time for an explosive branching process. Consider a Galton-Watson branching process with offspring distribution in the domain of attraction of a stable law of index $0<\ga<1$.  Let $X_n$ denote the coalescence time of any two individuals choosing at random at generation $n$. Then it is shown in  \cite{Athreya12} that  $\lim_{n\ra \infty}P\{n-X_n \leq k\}$ exists and can be calculated explicitly through $PD(\ga^k,0)$. In this case,  $\alpha=\ga^k$ converging to zero corresponds to $k$  converging to infinity.

       There have been intensive studies of the asymptotic behaviour for the Poisson-Dirichlet distribution and the Pitman-Yor process in recent years with motivations from probability theory, population genetics, and Bayesian statistics (see \cite{Feng10} and the references therein). 
       The results in this paper not only generalize some earlier results but, more importantly, reveal  some surprising new structures.        
          
    The paper is organized as follows. In Section 2, we review the subordinator representation for $PD(\alpha,0)$.  Section 3 contains the law of large numbers, fluctuation, and moderate  deviations associated with $PD(\alpha,0)$ as $\alpha$ converges to zero or one.  In Section 4, we establish the large deviation principle for $\Xi_{\alpha,0,\nu}$ under the limit of $\alpha$ converging to one. We finish the paper in Section 5 with some concluding remarks.        
      
      \section{Subordinator Representation}
      
      For any $0<\alpha<1$, let $\rho_t$ be the stable subordinator with index $\alpha$ and L\'evy measure
      \[
      \Lambda_{\alpha}(d\,x)= \frac{\alpha}{\Gamma(1-\alpha)}x^{-(1+\alpha)}d\, x, \ \  x >0.
      \]
      
      The boundary case  $\alpha=1$ corresponds  to the straight line $\rho_t=t$. When $\alpha$ converges to zero, $\rho_t$ becomes a killed subordinator with killing rate one (\cite{Ber96}).
      
      For any $t >0$, let $J_1(\rho_t)\geq J_2(\rho_t)\geq \cdots$ denote the jump sizes of $\rho_t$ over the interval $[0,t]$. Then the following representation holds. 
      
      \begin{theorem}\label{pre-t1}{\rm (Perman, Pitman, and Yor \cite{PPY92})}
      For any $t >0$, the law of 
      \be\label{pre-q1}
      (\frac{J_1(\rho_t)}{\rho_t},\frac{J_2(\rho_t)}{\rho_t},\ldots)
      \ee
      is $PD(\alpha,0)$.
      \end{theorem}
      
      For any $n \geq 1$, let $Z_n =\Lambda_{\alpha}(J_n(\rho_1), \infty)$. Then $Z_1 <Z_2<\ldots$ and $Z_1, Z_2-Z_1, Z_3-Z_2,\ldots$ are i.i.d. exponential random variables with parameter $1$. 
      Noting that $\Lambda_{\alpha}(x,\infty)= \frac{x^{-\alpha}}{\Gamma(1-\alpha)}$, it follows that
      \be\label{q1'}
      \frac{J_n(\rho_1)}{\rho_1}= \frac{Z_n^{-1/\alpha}}{\sum_{i=1}^\infty Z_i^{-1/\alpha}}
      \ee
      and 
      \be\label{q2'}
      \rho_1= \Gamma(1-\alpha)^{-1/\alpha} \sum_{i=1}^\infty Z_i^{-1/\alpha}.
      \ee
        Thus by Theorem~\ref{pre-t1} the law of
        
        \be\label{pre-q2}
        (\frac{Z_1^{-1/\alpha}}{\sum_{i=1}^{\infty}Z_i^{-1/\alpha}}, \frac{Z_2^{-1/\alpha}}{\sum_{i=1}^{\infty}Z_i^{-1/\alpha}}, \ldots)
        \ee
        is $PD(\alpha,0)$.   
             
      \section{Limit Theorems for $PD(\alpha,0)$}
      
     Let $${\bf P}(\alpha,0)=(P_1(\alpha,0), P_2(\alpha,0), \ldots)$$ and 
     $$\varphi_2({\bf P}(\alpha,0)):=\sum_{i=1}^{\infty}P_i^2(\alpha,0).$$ 
     A direct application of Pitman's sampling formula (\cite{Pitman1992}, \cite{Pitman06}) leads to
      \[
     \mathbb{E}_{\alpha,0}[\varphi_2({\bf P}(\alpha,0))] = 1-\alpha.      
      \]
      This implies that ${\bf P}(\alpha,0)$ converges in probability to $(1,0,\ldots)$ and $(0,0,\ldots)$ as $\alpha$ converges to $0$ and $1$, respectively. The objective of this section is to obtain more detailed information associated with these limits including fluctuation and large deviations.

\subsection{Convergence and Limit}

 For any $n \geq 1$, set
        \[
        P_n(\alpha,0)=\frac{Z_n^{-1/\alpha}}{\sum_{i=1}^{\infty}Z_i^{-1/\alpha}}        
        \]

Let $0<\ga(\alpha) \leq 1$ and $\iota(\alpha)>0$ be such that
\be\label{scale-e1}
 \lim_{\alpha \ra 0}\frac{\ga(\alpha)}{\alpha}=c_1 \in [0,+\infty]
\ee
and 
\be\label{scale-e2}
\lim_{\alpha \ra 1}\frac{\iota(\alpha)}{\Gamma(1-\alpha)} =c_2 \in [0,\infty).
\ee

\begin{theorem}\label{sec3-t1}
Let 
\[
{\bf P}^{\ga(\alpha)}(\alpha,0)=(P_1^{\ga(\alpha)}(\alpha,0), P_2^{\ga(\alpha)}(\alpha,0), \ldots).
\]
If $c_1$ is  finite, then  
 ${\bf P}^{\ga(\alpha)}(\alpha,0)$ converges  almost surely to $(1, (\frac{Z_1}{Z_2})^{c_1},(\frac{Z_1}{Z_3})^{c_1},\ldots)$ as $\alpha$ converges to $0$.  If $c_1=\infty$, then  ${\bf P}^{\ga(\alpha)}(\alpha,0)$ converges  to $(1,0,\ldots)$ in probability as $\alpha$ converges to $0$.
\end{theorem}
\proof  Set
\[
\tilde{\bf Z}= (Z_1^{-1}, Z_2^{-1}, \ldots).
\]

Then we have  
\[
(\sum_{i=1}^{\infty}Z_i^{-1/\alpha})^{\alpha}= ||\tilde{\bf Z}||_{1/\alpha}.
\]

When $\alpha$ approaches to zero, $||\tilde{\bf Z}||_{1/\alpha}$ converges almost surely to $||\tilde{\bf Z}||_{\infty}=Z^{-1}_1$. This implies that
\[
{\bf P}^{\alpha}(\alpha,0)=(\frac{Z^{-1}_1}{||\tilde{\bf Z}||_{1/\alpha}}, \frac{Z_2^{-1}}{||\tilde{\bf Z}||_{1/\alpha}}, \ldots)
\]
converges almost surely to $(1, \frac{Z_1}{Z_2}, \frac{Z_1}{Z_3}, \ldots)$ as $\alpha$ converges to zero. Writing ${\bf P}^{\ga(\alpha)}(\alpha,0)$ as
\[
((\frac{Z^{-1}_1}{||\tilde{\bf Z}||_{1/\alpha}})^{\ga(\alpha)/\alpha},(\frac{Z_2^{-1}}{||\tilde{\bf Z}||_{1/\alpha}})^{\ga(\alpha)/\alpha}, \ldots).
\]
Then by  continuity we obtain that ${\bf P}^{\ga(\alpha)}(\alpha,0)$ converges almost surely to $(1, (\frac{Z_1}{Z_2})^{c_1},(\frac{Z_1}{Z_3})^{c_1},\ldots)$ as $\alpha$ converges to zero. If $c_1=\infty$, then for any $M\geq 1$ one has $\frac{\ga(\alpha)}{\alpha} > M$ for small enough $\alpha$. Thus for any $n > 1$
\[
\lim_{\alpha \ra 0}P_n^{\ga(\alpha)}(\alpha,0) \leq \lim_{\alpha \ra 0}P_n^M(\alpha,0) =(\frac{Z_1}{Z_n})^M. 
\]  
Since $M$ is arbitrary,  we obtain
\[
\lim_{\alpha \ra 0}P_n^{\ga(\alpha)}(\alpha,0)=0, \ a.s., \ n>1.
\]
Finally for $n=1$, we have
$$P_1(\alpha,0)\leq P_1^{\ga(\alpha)}(\alpha,0)\leq 1.$$
Noting that
\[
\mathbb{E}[P_1(\alpha,0)] \leq \mathbb{E}[\varphi_2({\bf P}(\alpha,0))]=1-\alpha.
\]
It follows that $P_1(\alpha,0)$ converges to $1$ in probability which implies that $P_1^{\ga(\alpha)}(\alpha,0)$ converges to one in probability.
\hfill  $\Box$

\begin{theorem}\label{sec3-t2}
As $\alpha$ converges to $1$, $\iota(\alpha){\bf P}(\alpha,0)$ converges in probability to $c_2 \, (Z_1^{-1}, Z_2^{-1}, \ldots)$.
\end{theorem}
\proof Let $S_{\alpha}= \rho_1^{-\alpha}$. Then the law of $S_{\alpha}$ is the Mittag-Leffler distribution with density function 
        \[
        g_{\alpha}(s)=\sum_{k=0}^{\infty}\frac{(-s)^k}{k!}\Gamma(\alpha k+\alpha+1)\frac{\sin(\alpha k\pi )}{\alpha k \pi} 
        \]
        and  
        \beqn
         &&\sum_{i=1}^\infty Z_i^{-1/\alpha}= (\frac{S_{\alpha}}{\Gamma(1-\alpha)})^{-1/\alpha}\label{pre-q3}\\
         && \mathbb{E}[S_{\alpha}^r]= \frac{\Gamma(r+1)}{\Gamma(\alpha r+1)}, \ r>-1\label{pre-q4}.
        \eeqn
        This implies that 
         \beq
        \mathbb{E}[(S_{\alpha}-1)^2]&=& \frac{2}{\Gamma(2\alpha+1)}- \frac{2}{\Gamma(\alpha +1)} +1\\
        &\ra& 0,\ \ \alpha \ra 1.  
        \eeq
       Hence $S_{\alpha}$ converges to $1$ in probability as $\alpha$ converges to $1$.
        
         By \rf{q2'}, one has 
\beq
\iota(\alpha){\bf P}(\alpha,0)&=& \frac{\iota(\alpha)}{\Gamma(1-\alpha)}\Gamma(1-\alpha)^{1-\frac{1}{\alpha}}((\frac{Z_1}{S_{\alpha}})^{-1/\alpha}, (\frac{Z_2}{S_{\alpha}})^{-1/\alpha},\ldots)\\
&=&\frac{\iota(\alpha)}{\Gamma(1-\alpha)}\Gamma(1-\alpha)^{1-\frac{1}{\alpha}}\exp\{\frac{1}{\alpha}\log S_{\alpha}\}(Z_1^{-1/\alpha},Z_2^{-1/\alpha}, \ldots).\eeq
Since $S_{\alpha}$ converges to one in probability and $(Z_1^{-1/\alpha}, Z_2^{-1/\alpha}, \ldots)$ converges to $(Z_1^{-1}, Z_2^{-1}, \ldots)$ almost surely as $\alpha$ converges to one, we conclude that $\iota(\alpha){\bf P}(\alpha,0)$ converges  to $c_2\, (Z_1^{-1}, Z_2^{-1}, \ldots)$ in probability.

\hfill $\Box$

\subsection{Large Deviations}

  In this section we consider the large deviations associated with  the deterministic limits obtained in Theorem~\ref{sec3-t1}. In comparison with the large deviations associated with ${\bf P}(\alpha,0)$ these results can be viewed as moderate deviations for ${\bf P}(\alpha,0)$.  We prove these results  through a series of lemmas.

For any $n \geq 1$ let
\[
R_n =\frac{P_{n+1}(\alpha,0)}{P_n(\alpha,0)}.
\]
Then  $\{R_n:n \geq1\}$ is a sequence of  independent beta random variables with each $R_n$ having the  $beta(n\alpha,0)$ distribution (Proposition 8 in \cite{PitmanYor97}).  
\begin{lemma}\label{ldp-l1}
Let ${\bf R}^{\ga(\alpha)}=(R_1^{\ga(\alpha)}, R_2^{\ga(\alpha)}, \ldots)$. As $\alpha$ converges to $0$, 
 large deviation principles hold for ${\bf R}^{\ga(\alpha)}$ on space $[0,1]^{\infty}$ with respective speeds and  rate functions 
 $(\frac{\alpha}{\ga(\alpha)}, J_1(\cdot))$ and $(\log \frac{\ga(\alpha)}{\alpha}, J_2(\cdot))$ depending on whether $c_1=0$ or $c_1=\infty$, where

  \[
J_1({\bf x})=\left\{ \begin{array}{ll}
\sum_{n=1}^\infty n\log \frac{1}{x_{n}},&  x_n>0 \ \mb{for all}\ n>1,\\
+\infty,& otherwise.
\end{array}
\right.
\]
 and  
   
   \[
J_2({\bf x})=\#\{n\geq 1: x_n >0\}.
\]    
\end{lemma}
\proof Assume that $c_1=0$. For any $n\geq 1$ and any $x$ in $[0,1]$, one has
\beq
n\log x \leq \lim_{\delta\ra 0}\liminf_{\alpha \ra 0}\frac{\ga(\alpha)}{\alpha}\log \mathbb{P}\{|R_n^{\ga(\alpha)}-x|<\delta\}\\
n\log x \geq \lim_{\delta\ra 0}\limsup_{\alpha \ra 0}\frac{\ga(\alpha)}{\alpha}\log \mathbb{P}\{|R_n^{\ga(\alpha)}-x|\leq\delta\}.
\eeq
This combined with the compactness of $[0,1]$ implies that $R_n^{\ga(\alpha)}$ satisfies a large deviation principle on $[0,1]$ with speed $\frac{\alpha}{\ga(\alpha)}$ and rate function $n\log x$.   Similarly for $c_1=\infty$, we have  
\beq
-\chi_{\{x>0\}} \leq \lim_{\delta\ra 0}\liminf_{\alpha \ra 0}(\log\frac{\ga(\alpha)}{\alpha})^{-1}\log \mathbb{P}\{|R_n^{\ga(\alpha)}-x|<\delta\}\\
-\chi_{\{x>0\}} \geq \lim_{\delta\ra 0}\limsup_{\alpha \ra 0}(\log\frac{\ga(\alpha)}{\alpha})^{-1}\log \mathbb{P}\{|R_n^{\ga(\alpha)}-x|\leq\delta\}.
\eeq
These combined with the independence of  $R_1, R_2, \dots$ imply the large deviations for ${\bf R}^{\ga(\alpha)}$.\hfill $\Box$

\begin{lemma}\label{ldp-l2}
There exists  $\delta \geq 1$ such that for any $\la <\delta$ 
\be\label{ldp-e5}
\mathbb{E}[\exp\{\la(1-\alpha)(P^{-1}_1(\alpha,0)-1)\}]= (1+A_{\la,\alpha})^{-1}<\infty
\ee
where
 \[
A_{\la,\alpha}=\alpha \int_0^1(1-e^{\la(1-\alpha)z})z^{-(1+\alpha)}d\,z.
\]
\end{lemma}

\proof  Clearly $A_{\la,\alpha}$ is nonnegative for $\la \leq 0$, and converges to negative infinity as $\la$ tends to positive infinity.  It is known (equation (77) in \cite{Kingman75}) that 
\be\label{ldp-e3}
\mathbb{E}[\exp\{\la(1-\alpha)(P^{-1}_1(\alpha,0)-1)\}]= (1+A_{\la,\alpha})^{-1}<\infty
\ee
for $\la \leq 0$. For $\la>0$,  we have
\beqn\label{ldp-e4}
A_{\la, \alpha}&=& (1-\la)e^{\la(1-\alpha)}-1 +\la^2(1-\alpha)\int_0^1 z^{1-\alpha}e^{\la(1-\alpha)z}d\,z\\
&\geq& (1-\la) e^{\la(1-\alpha)} -1 +\la^2(1-\alpha)\int_0^1z^{1-\alpha}e^{\la(1-\alpha)z}d\,z.\nn
 \eeqn
 
 If we define
 $$\la_{\alpha}= \sup\{\la>0:A_{\la,\alpha} +1>0\},$$
then $\la_{\alpha}\geq 1$ by \rf{ldp-e4}  and 
\[
\delta=\inf\{\la_{\alpha}: 0<\alpha<1\}\geq 1.
\]

 By Campbell's theorem \rf{ldp-e5} holds for any $\la<\delta$. 

\hfill $\Box$
\begin{lemma}\label{ldp-l3}
Let $\ep>0$ be arbitrarily given. If $c_1=0$, then 

\be\label{ldp-e7}
\limsup_{\alpha \ra 0}\frac{\ga(\alpha)}{\alpha}\log \mathbb{P}\{|P_1^{\ga(\alpha)}(\alpha,0)-1|>\ep\}=-\infty.
\ee
If $c_1=\infty$ and
\be\label{ldp-e11}
\lim_{\alpha \ra 0}\ga(\alpha)=0,
\ee
then
\be\label{ldp-e6}
\limsup_{\alpha \ra 0}\frac{1}{\log \frac{\ga(\alpha)}{\alpha}}\log \mathbb{P}\{|P_1^{\ga(\alpha)}(\alpha,0)-1|>\ep\}=-\infty.
\ee
\end{lemma}
\proof   Since the limit involves only small $\alpha$, we may assume that $0<\alpha<1/2$ and $0< \ep <1/2$. Let $\delta$ be as  in Lemma~\ref{ldp-l2} and set  $\delta_1=\delta/4$.  By direct calculation we obtain that
\beqn\label{ldp-e8}
\mathbb{P}\{|P^{\ga(\alpha)}_1(\alpha,0)-1|>\ep\}&=&\mathbb{P}\{P^{-1}_1(\alpha,0)-1\geq (1-\ep)^{-1/\ga(\alpha)}-1\}\nn\\
&\leq & \mathbb{E}[e^{\delta_1(P_1^{-1}(\alpha,0)-1)}]e^{-\delta_1[(1-\ep)^{-1/\ga(\alpha)}-1]}\\
&\leq & (1+A_{\delta_1, \alpha})^{-1}e^{-\delta_1[(1-\ep)^{-1/\ga(\alpha)}-1]}.\nn
\eeqn

It follows from \rf{ldp-e4} that
\be\label{ldp-e9}
\lim_{\alpha\ra 0} (1+A_{\delta_1, \alpha}) =1. 
\ee
 
 If $c_1=0$, then 
 \beqn\label{ldp-e10}
 &&\limsup_{\alpha\ra 0}\frac{(1-\ep)^{-1/\ga(\alpha)}-1}{\frac{\alpha}{\ga(\alpha)}}= \limsup_{\alpha\ra 0}\frac{(1-\ep)^{-1/\ga(\alpha)}}{\frac{\alpha}{\ga(\alpha)}}\nn\\
  &&\hspace{1cm} = \limsup_{\alpha \ra 0}\exp{\{\frac{1}{\ga(\alpha)}[\log\frac{1}{(1-\ep)}+\ga(\alpha)\log \ga(\alpha)-\frac{\ga(\alpha)}{\alpha} \alpha \log \alpha] \}}\\
 && \hspace{1cm}= \infty.  \nn
 \eeqn

Next assume that $c_1=\infty$ and \rf{ldp-e11} hold. For any $0<\ep<1/2$, $(1-\ep)^{1/\ga(\alpha)}$ converges to zero as $\alpha$ tends to zero. Hence for any $k\geq 1$, one can find $\alpha_k>0$ such that for all $0<\alpha<\alpha_k$
\[
\mathbb{P}\{|P_1^{\ga(\alpha)}(\alpha,0)-1|>\ep\}\leq \mathbb{P}\{P_1(\alpha,0)< \frac{1}{k}\}.
\]
By the large deviation principle for $P_1(\alpha,0)$ in \cite{Feng09}, we obtain that
\beq
\limsup_{\alpha \ra 0}\frac{1}{\log\frac{1}{\alpha}}\log \mathbb{P}\{|P_1^{\ga(\alpha)}(\alpha,0)-1|>\ep\}&\leq& \limsup_{\alpha \ra 0}\frac{1}{\log\frac{1}{\alpha}}\log \mathbb{P}\{P_1(\alpha,0)< \frac{1}{k}\}\\
&\leq& -(k-1).
\eeq
 
 Noting that $\ga(\alpha)<1$ and $k$ is arbitrary it follows that
 
 \beqn\label{ldp-e12}
 &&\limsup_{\alpha\ra 0}\frac{1}{\log\frac{\ga(\alpha)}{\alpha}}\log \mathbb{P}\{|P_1^{\ga(\alpha)}(\alpha,0)-1|>\ep\}\hspace{6cm}\nn\\
 &&\hspace{3.5cm}\leq \limsup_{\alpha\ra 0}\frac{1}{\log\frac{1}{\alpha}}\log \mathbb{P}\{|P_1^{\ga(\alpha)}(\alpha,0)-1|>\ep\}  \\
  &&\hspace{3.5cm}\leq\lim_{k\ra \infty }\limsup_{\alpha\ra 0}\frac{1}{\log\frac{1}{\alpha}}\log \mathbb{P}\{P_1(\alpha,0)\leq \frac{1}{k}\} \nn\\
 &&\hspace{3.5cm} = -\infty.\nn
 \eeqn 

 Putting together  \rf{ldp-e8}-\rf{ldp-e12}, we get \rf{ldp-e7} and \rf{ldp-e6}.
  
\hfill $\Box$

\begin{theorem}\label{ldp-t1}
Let $\ga(\alpha)$ satisfy \rf{scale-e1}, and set
\[
\nabla=\{{\bf x}=(x_1,x_2,\ldots): 1\geq x_1\geq x_2\geq\cdots\geq 0\}.
\]
 Then the followings hold as $\alpha$ converges to $0$.

{\rm (i)} If $c_1 =0$, then the family $\{{\bf P}^{\ga(\alpha)}(\alpha,0): 0<\alpha <1\}$ satisfies a large deviation principle on space $\nabla$ with speed $\frac{\alpha}{\ga(\alpha)}$ and rate function
\be\label{ldp-e2}
I_1({\bf x})=\left\{ \begin{array}{ll}
\sum_{n=1}^\infty n\log \frac{x_{n}}{x_{n+1}},& x_1=1, x_n>0 \ \mb{for all}\ n>1,\\
+\infty,&otherwise.
\end{array}
\right.
\ee
{\rm (ii)} If $c_1 =\infty$ and \rf{ldp-e11} holds, then the family $\{{\bf P}^{\ga(\alpha)}(\alpha,0): 0<\alpha <1\}$ satisfies a large deviation principle on space $\nabla$  with speed $\log\frac{\ga(\alpha)}{\alpha}$ and the rate function

\be\label{ldp-e1}
I_2({\bf x})=\left\{ \begin{array}{ll}
n-1,& x_1=1, x_n>0, x_k=0, k> n,\\
+\infty,& otherwise.
\end{array}
\right.
\ee
\end{theorem}

\proof Writing ${\bf P}^{\ga(\alpha)}$ in terms of ${\bf R}^{\ga(\alpha)}$ we have
\[
{\bf P}^{\ga(\alpha)}= P_1^{\ga(\alpha)}(\alpha,0)(1, R^{\ga(\alpha)}_1,R^{\ga(\alpha)}_1 R^{\ga(\alpha)}_2, \ldots).
\] 

By Lemma~\ref{ldp-l3}, $P_1^{\ga(\alpha)}(\alpha,0)$ is exponentially equivalent to one. Hence by lemma 2.1 in \cite{FengGao08} 
 $(1, R^{\ga(\alpha)}_1,R^{\ga(\alpha)}_1 R^{\ga(\alpha)}_2, \ldots)$ and $ {\bf P}^{\ga(\alpha)}$ have the same large deviation principle. Define
\[
\psi: [0,1]^{\infty}\lra \nabla,  \ \  (x_1,x_2,\ldots)\ra (1, x_1,x_1x_2, \ldots). 
\]
Then $\psi$ is clearly continuous and $(1, R^{\ga(\alpha)}_1,R^{\ga(\alpha)}_1 R^{\ga(\alpha)}_2, \ldots)= \psi({\bf R}^{\ga(\alpha)})$.  Noting that
\[
I_i({\bf x})= \inf\{J_i({\bf y}): \psi({\bf y})={\bf x}\}, i=1,2,
\]
the theorem follows from Lemma~\ref{ldp-l1} and the contraction principle.

\hfill $\Box$


\section{Asymptotic Behaviour of $\Xi_{\alpha,0,\nu}$}

Recall that the REM has configuration space $S_{N}=\{-1,1\}^{N}$ and the Hamiltonian given by  a family of $i.i.d.$ normal random variables with mean $0$ and variance $N$
$$\{H_{N}(\sigma)\mid \sigma\in S_{N}\}.$$
 The Gibbs measure $G_{N}(\sigma)$ at temperature $T$ is given by
 $$
 Z_{N}^{-1}\exp\{-\beta H_{N}(\sigma)\},
 $$
 where $\beta=1/T$ and $Z_{N}=\sum_{\sigma\in S_{N}}\exp\{-\beta H_{N}(\sigma)\}.$ By making the change of variable
 $$
 r_{N}(\sigma)=1-\sum_{i=1}^{N}(1-\sigma_{i})2^{-i-1},
 $$
we can regard $[0,1]$ as the new configuration space.  The corresponding Gibbs measure has the form
$$
\mu_{N}^{T}(d\,x)=\sum_{\sigma\in S_{N}}\delta_{r_{N}}(d\,x)G_{N}(\sigma).
$$
 As $N\to\infty$, the limiting Gibbs measure $\mu^{T}=\lim_{N\to\infty}\mu_{N}^{T}$ exhibits phase transition at the critical temperature $T_{c}=\sqrt{2\log2}$. More specifically, by Theorems 9.3.1 and 9.3.4 in \cite{Bov06}, we have
$$
\mu^{T}=\begin{cases}
\nu,& \mbox{ if}\ T\geq T_{c}\\
\Xi_{\alpha,0,\nu}, & \mbox{ if } T<T_{c}.
\end{cases}
$$ 

Thus a phase transition occurs when the temperature crosses the critical value between high temperature and low temperature regimes. The low temperature regime has a rich structure.  The transition from the low temperature regime to the critical temperature regime corresponds to $\alpha$ tending to one from below. The goal of this section is to understand the microscopic behaviour of this transition through the establishment of a large deviation principle for $\Xi_{\alpha,0,\nu}$.

\subsection{Estimates for Stable Subordinator}

Recall that $\rho_t$ be the stable subordinator with index $0<\alpha<1$. For $t=1$, the following holds.

\begin{lemma}{\rm (\cite{Pollard46}, \cite{Kan75})}
 The distribution function of $\rho_{1}^{\frac{\alpha}{1-\alpha}}$ has two integral representations:
 \begin{equation}
 F(x)=\mathbb{P}\{\rho_{1}^{\frac{\alpha}{1-\alpha}}\leq x\}=\frac{1}{\pi}\int_{0}^{\pi}e^{-\frac{A(u)}{x}}du,\label{first}
 \end{equation}
 where $A(u)$ is the Zolotarev's function defined as
 $$
 A(u)=\left\{\frac{\sin^{\alpha}(\alpha u)\sin^{1-\alpha}((1-\alpha)u)}{\sin u}\right\}^{\frac{1}{1-\alpha}}.
 $$
 The distribution function of $\rho_{1}$ is thus
 $F(x^{\frac{\alpha}{1-\alpha}}).
 $
 The density function of $\rho_{1}$ has the following representation
 \begin{equation}
 \phi_{\alpha}(t)=\frac{1}{\pi}\int_{0}^{\infty}e^{-tu}e^{-u^{\alpha}\cos \pi\alpha}\sin (u^{\alpha}\sin \pi\alpha)du\label{second}
 \end{equation}
\end{lemma}
Applying these representations, we obtain the following estimations. 

\begin{theorem}\label{fgz-t1}
For any given $1>\delta>0$, we have
\be\label{fgz-1}
\lim_{\alpha\to1}(1-\alpha)\log\log\frac{1}{\mathbb{P}\{\rho_{1}<1-\delta\}}=\lim_{\alpha\to1}(1-\alpha)\log\log\frac{1}{\mathbb{P}\{\rho_{1}\leq 1-\delta\}}=\log\frac{1}{1-\delta}
\ee
and
\be\label{fgz-2}
\lim_{\alpha\to1}\frac{1}{\log\frac{1}{1-\alpha}}\log \mathbb{P}\{\rho_{1}>1+\delta\}=\lim_{\alpha\to1}\frac{1}{\log\frac{1}{1-\alpha}}\log \mathbb{P}\{\rho_{1}\geq1+\delta\}=-1.
\ee
\end{theorem}
\begin{proof}
For any $u\in (0,\pi), v\in (0,1)$, one has 
\beq
\frac{d [v \cot(v u)-\cot u]}{dv}&=&\frac{1}{2\sin^{2}(v u)}(\sin(v u)-2v u)\\
&\leq& \frac{1}{2\sin^{2}(v u)}(\sin(v u)-v u)\leq 0\eeq
which implies that
$$ 
 \frac{d\log\frac{\sin(v u)}{\sin u}}{du}=v\cot(v u)-\cot u\geq 0.
 $$
 Hence 
 $$
 A(u)=\exp\{\alpha\log\frac{\sin(\alpha u)}{\sin u}+(1-\alpha)\log\frac{\sin ((1-\alpha)u)}{\sin u}\}
$$ 

\noindent is nondecreasing in $u$. Further more it follows from direct calculation that 
\[
\lim_{u\to0}A(u)=(1-\alpha)\alpha^{\frac{\alpha}{1-\alpha}}\quad \lim_{u\to\pi}A(u)=\infty.
\]

\noindent Therefore, applying the representation \rf{first} we get that for any $\epsilon>0$
\beq
&&\frac{\pi-\epsilon}{\pi}\exp\left\{-\frac{A(\pi-\epsilon)}{(1-\delta)^{\frac{\alpha}{1-\alpha}}}\right\}\\
&&\leq\frac{1}{\pi}\int_{0}^{\pi-\epsilon}e^{-\frac{A(u)}{(1-\delta)^{\frac{\alpha}{1-\alpha}}}}du
=\mathbb{P}\{\rho_1 \leq 1-\delta\}\\
&&\leq \exp\left\{-\frac{A(0)}{(1-\delta)^{\frac{\alpha}{1-\alpha}}}\right\}.\eeq
This implies that 
\beq
\log\frac{1}{1-\delta}&\leq&\liminf_{\alpha\to1}(1-\alpha)\log\log\frac{1}{\mathbb{P}\{\rho_{1}\leq 1-\delta\}}\\
&=& \liminf_{\alpha\to1}(1-\alpha)\log\log\frac{1}{\mathbb{P}\{\rho_{1}<1-\delta\}}
\eeq
\beq
&&\limsup_{\alpha\to1}(1-\alpha)\log\log\frac{1}{\mathbb{P}\{\rho_{1}<1-\delta\}}\\
&&=\limsup_{\alpha\to1}(1-\alpha)\log\log\frac{1}{\mathbb{P}\{\rho_{1}\leq 1-\delta\}}\\
&& \leq \log(\frac{1}{1-\delta})
\eeq
and thus \rf{fgz-1} holds.

\noindent To prove \rf{fgz-2}, we apply  (\ref{second}) and get
\begin{align*}
\mathbb{P}\{\rho_{1}>1+\delta\}=&\mathbb{P}\{\rho_{1}\geq1+\delta\}\\
=&\frac{1}{\pi}\int_{1+\delta}^{\infty}\int_{0}^{\infty}u^{-1}e^{-(1+\delta)u}e^{-u^{\alpha}\cos\pi\alpha}\sin(u^{\alpha}\sin\pi\alpha)du dt\\
=&\frac{\sin\pi\alpha}{\pi}\int_{0}^{\infty}u^{-(1-\alpha)}e^{-\delta u} \left[e^{-u-u^{\alpha}\cos\pi\alpha}\frac{\sin(u^{\alpha}\sin\pi\alpha)}{u^{\alpha}\sin\pi\alpha}\right]du
\end{align*}

\noindent Noting that
$
\frac{\sin(u^{\alpha}\sin\pi\alpha)}{u^{\alpha}\sin\pi\alpha}$ is bounded and 
$$
\lim_{\alpha \ra 1}\frac{\sin \pi\alpha}{\pi(1-\alpha)}
=1,$$
it follows that \rf{fgz-2} holds.

\end{proof}

\begin{theorem}\label{fgz-t2}
The family $\{\rho_1: 0<\alpha<1\}$ satisfies a large deviation principle on $(0,\infty)$ as $\alpha$ tends to one  with speed $-\log(1-\alpha)$  and rate function {\rm (not good in this case)}
\be\label{fgz-3}
J(x)=\left\{ \begin{array}{ll}
1,& x>1,\\
0,& x=1\\
+\infty,& otherwise.
\end{array}
\right.
\ee
\end{theorem}
\proof  Let $A$ be a closed set in $(0,\infty)$. If $A$ contains $1$, then $\inf_{x\in A}J(x)=0$ and the upper estimate holds. If $A$ does not contain $1$, then one can find $0<a<1<b$ such that $A$ is either a subset of $(0,a]$, a subset of $[b,\infty)$ or a subset or $(0,a]\cup [b,\infty)$.  For each case we can apply Theorem\ref{fgz-t1} to obtain the upper estimate.

\noindent The proof for lower estimates goes as follows. Let $B$ be any open set. If $B$ intersects with $[0,1)$, then the lower estimates are trivial. If $B$ does not intersect with $[0,1)$, then $B$ can not contain $1$. Hence one can find $1<a<b<\infty$ such that $(a,b)\subset B$ and
\beq
\mathbb{P}\{\rho_{1}\in B\}&\geq &\mathbb{P}\{\rho_{1}\in (a,b)\}\\
&\geq & \frac{b-a}{\pi}\int_{0}^{\infty}u^{-1}e^{-bu}e^{-u^{\alpha}\cos\pi\alpha}\sin(u^{\alpha}\sin\pi\alpha)du
\eeq
which implies that
  
\[
\liminf_{\alpha\ra 1}\frac{1}{-\log(1-\alpha)}\log \mathbb{P}\{\rho_{1}\in B\}\geq-1 =-inf_{x \in B}J(x)
 \]
\hfill $\Box$

For any $n \geq 1$, let $\tau_1, \ldots, \tau_{n+1}$ be independent copies of $\rho_1$. Set
\[
\sigma_i=\frac{\tau_i}{\tau_1}, \ i=2, \ldots,n+1.
\]
Set
\[
\tilde{\sigma}_n =\min\{\sigma_i:2\leq i\leq n+1\}\] 
and let $r_n$ denote the frequency of $\tilde{\sigma}_n$ among $\{\sigma_i\}_{i=2,\ldots,n+1}$. Define

\be\label{fgz-4}
J_n(u_1,\ldots, u_n)=\left\{ \begin{array}{ll}
n+1-r_n,& \tilde{\sigma}_n <1,\\
n-r_n,& \tilde{\sigma}_n=1,\\
n, &\tilde{\sigma}_n >1.
\end{array}
\right.
\ee
Clearly $J_n(\cdot)$ is a rate function on $(0,\infty)^n$.

\begin{theorem}\label{fgz-t3}
The family $\{(\sigma_2,\ldots,\sigma_{n+1}): 0<\alpha<1\}$ satisfies a large deviation principle on $(0,\infty)^n$ with speed $-\log(1-\alpha)$  and rate function $J_n(\cdot)$ as $\alpha$ tends to one.
\end{theorem}

\proof  Note that the map
\[
\Phi: (0,\infty)^{n+1} \ra (0,\infty)^n, \ (x_1,\ldots,x_{n+1})\ra (\frac{x_2}{x_1},\ldots, \frac{x_{n+1}}{x_1}) 
\]
is clearly continuous. It follows from the contraction principle that large deviation upper and lower estimates hold for the family $\{(\sigma_2,\ldots,\sigma_{n+1}): 0<\alpha<1\}$ with the bounds given by the function
\[
\tilde{J}_n(u_1,\ldots,u_n)=\inf\{\sum_{i=1}^{n+1} J(x_i): x_{j+1}=u_{j}x_1, j=1,\ldots, n\}.
\]
Since $J(x)=\infty$ for $x$ in $(0,1)$, it follows that
\[
\tilde{J}_n(u_1,\ldots,u_n)=\inf\{\sum_{i=1}^{n+1} J(x_i): x_1\geq 1, \ x_{j+1}=u_{j}x_1\geq 1,\  j=1,\ldots, n\}=J_n(u_1,\ldots,u_n)
\]
and the theorem follows.

\hfill $\Box$

\noindent {\bf Remark}. The contraction principle used in Theorem~\ref{fgz-t3} does not lead to a large deviation principle in general due to the fact that the starting rate function is not good.  But here and later on, direct calculations show that the upper and lower bounds are all given by rate functions. 

\subsection{Large Deviations for $\Xi_{\alpha,0,\nu}$}

Let $M_1([0,1])$ denote the space of probabilities on $[0,1]$ equipped with the weak topology.  For any $\mu$ in $M_1([0,1])$ define
\[
{\cal I}(\mu)=\begin{cases}
0, & \mu=\nu\\
n,&\mu=\sum_{i=1}^np_i\delta_{x_i}+(1-\sum_{i=1}^n p_i)\nu\\
\infty,& \mbox{ otherwise}.
\end{cases}
\]

\noindent The main result of this subsection is 
\begin{theorem}\label{fgz-t4}
The family $\{\Xi_{\alpha,0,\nu}: 0<\alpha<1\}$ satisfies a large deviation principle on $M_1([0,1])$ with speed $-\log(1-\alpha)$  and good rate function ${\cal I}(\cdot)$ as $\alpha$ tends to one.
\end{theorem}

\noindent We prove this theorem through a series of lemmas.

 \begin{lemma}\label{fgz-l1}
 For any $n \geq 1$,  let $0=t_{0}<t_{1}<\cdots<t_{n}<t_{n+1}=1$ and  $B_{1},\cdots,B_{n+1}$ be a measurable partition of $[0,1]$ such that $\nu(B_{i})=t_{i}-t_{i-1}$. 
 Then 
 \beq
 &&(\Xi_{\alpha,0,\nu}(B_{1}),\cdots,\Xi_{\alpha,0, \nu}(B_{n+1}))\\
 &&\  \eqdist
\rho_1^{-1}(\rho_{t_{1}}, \rho_{t_{2}}-\rho_{t_{1}},\cdots,\rho_{t_{k}}-\rho_{t_{k-1}}, \rho_1-\rho_{t_{k}})\\
&& \  \eqdist (t_1^{1/\alpha}+\sum_{k=2}^{n+1}(t_k-t_{k-1})^{1/\alpha}\sigma_k)^{-1}(t_1^{1/\alpha},(t_2-t_1)^{1/\alpha}\sigma_2,\ldots, (1-t_n)^{1/\alpha}\sigma_{n+1})
 \eeq
 where $\eqdist$ denotes equality in distribution.
  \end{lemma} 
 \proof The first equality is from \cite{Pitman96} and the second equality follows from the independent increments of the stable subordinator and the equality
 $
 \rho_t \eqdist t^{1/\alpha}\rho_1.
 $
 
 \hfill $\Box$

 \begin{lemma}\label{fgz-l2}
 Let 
 \[
 \triangle_{n+1} :=\{(y_1,\ldots,y_{n+1}): y_i\geq 0, \sum_{k=1}^{n+1} y_k =1\}.
 \]
 
\noindent Then the  family $\{(\Xi_{\alpha,0,\nu}(B_{1}),\cdots,\Xi_{\alpha,0, \nu}(B_{n+1})): 0<\alpha<1\}$ satisfies a large deviation principle on $\triangle_{n+1}$ with speed $-\log(1-\alpha)$  and good rate function ${\cal I}_n(\cdot)$ as $\alpha$ tends to one, where
\[
{\cal I}_n(y_1,\ldots,y_{n+1})=(n+1)-\gamma(y_1,\ldots,y_{n+1})
\]
with
\[
\gamma(y_1,\ldots,y_{n+1})=\#\{1\leq i \leq n+1: \frac{y_i}{t_i-t_{i-1}}=\min\{\frac{y_k}{t_k-t_{k-1}}: 1\leq k \leq n+1\}\}.
\]
 \end{lemma}
  \proof  First note that the map
  \beq
 && H: [0,1]^n\times (0,\infty)^n \ra [0,1], \\
 &&  (v_1,\ldots,v_{n+1}; u_1, \ldots,u_n)\ra (v_1+\sum_{k=2}^{n+1}v_k u_{k-1})^{-1}(v_1, v_2 u_1,\ldots,v_{n+1}u_n)
  \eeq
  is continuous and $(\Xi_{\alpha,0,\nu}(B_{1}),\cdots,\Xi_{\alpha,0, \nu}(B_{n+1}))$ has the same distribution as
  \[
  H(t_1^{1/\alpha}, \ldots, (1-t_n)^{1/\alpha}; \sigma_2,\ldots, \sigma_{n+1}).
  \]
  
  Noting that $(t_1^{1/\alpha}, \ldots, (1-t_n)^{1/\alpha})$  satisfies a full large deviation principle with effective domain $(t_1, \ldots, (1-t_n))$. It follows from Theorem~\ref{fgz-t3}, the independence between $(t_1^{1/\alpha}, \ldots, (1-t_n)^{1/\alpha})$ and $(\sigma_2, \ldots, \sigma_{n+1})$  and the contraction principle that large deviation estimates hold for  $(\Xi_{\alpha,0,\nu}(B_{1}),\cdots,\Xi_{\alpha,0, \nu}(B_{n+1}))$  with upper and lower bounds given by the function
  \beq
  \tilde{\cal I}_n(y_1, \ldots, y_{n+1})&=&\inf\{J_n(u_1,\ldots,u_n): u_{i}\in (0,\infty),  u_i=\frac{t_1}{y_1}\frac{y_{i+1}}{t_{i+1}-t_{i}}, i=1, \ldots, n\}\\
  &=&\left\{ \begin{array}{ll}
n+1-\tilde{r}_n,& \min_{2\leq i \leq n+1}\{\frac{y_i}{t_i-t_{i-1}}\} <\frac{y_1}{t_1},\\
n-\tilde{r}_n,& \min_{2\leq i \leq n+1}\{\frac{y_i}{t_i-t_{i-1}}\} =\frac{y_1}{t_1},\\
n, &\min_{2\leq i \leq n+1}\{\frac{y_i}{t_i-t_{i-1}}\} >\frac{y_1}{t_1}.
\end{array}
\right.
  \eeq
   where $\tilde{r}_n$ is the frequency of   $\min_{2\leq i \leq n+1}\{\frac{y_i}{t_i-t_{i-1}}\}$  among $\frac{y_2}{t_2-t_1},\ldots, \frac{y_{n+1}}{1-t_n}$.  On the other hand,
   \[
   \gamma(y_1,\ldots, y_{n+1})=\left\{ \begin{array}{ll}
\tilde{r}_n,& \min_{2\leq i \leq n+1}\{\frac{y_i}{t_i-t_{i-1}}\} <\frac{y_1}{t_1},\\
\tilde{r}_n+1,& \min_{2\leq i \leq n+1}\{\frac{y_i}{t_i-t_{i-1}}\} =\frac{y_1}{t_1},\\
1, &\min_{2\leq i \leq n+1}\{\frac{y_i}{t_i-t_{i-1}}\} >\frac{y_1}{t_1}.
\end{array}
\right.
   \]
   Hence we obtain that $\tilde{\cal I}_n(\cdot)={\cal I}_n(\cdot)$. It remains to show that ${\cal I}_n(\cdot)$  is a good rate function. Since $\triangle_{n+1}$ is compact, it suffices to verify the lower semicontinuity of the ${\cal I}_n(\cdot)$. For any point $(y_1,\ldots,y_{n+1})$ in $\triangle_{n+1}$, let $\gamma(y_1, \ldots, y_{n+1})=m$. If the neighbourhood of $(y_1,\ldots,y_{n+1})$ is small enough, then the frequency of the minimum in each point inside the neighbourhood is at least $m$. Hence ${\cal I}(\cdot)$ is lower semicontinuous. 
 
  \hfill $\Box$

\begin{lemma}\label{fgz-l3}
 
\beqn\label{fgz-5}
  {\cal I}(\mu)&=&\sup\{{\cal I}_n(\mu([0,t_1]),\mu((t_1,t_2]),\ldots, \mu((t_n,1]):\\
&&  0=t_0<t_1<\cdots<t_n<t_{n+1}=1, n =1,2,\ldots\}.\nn
\eeqn

The supremum can be taken over all continuity points $t_1, \ldots, t_n$ of $\mu$.

 \end{lemma}
  \proof   We divide the proof into several cases. Let $\mu$ be any probability in $M_1([0,1])$. By  Lebesgue's Decomposition Theorem, one can write
  \[
  \mu=\la_1\mu_a+\la_2\mu_s+\la_3\mu_{ac}
  \]
  where $\mu_{a}$ is atomic, $\mu_s$ is singular with respect to $\nu$,  $\mu_{ac}$ is absolutely continuous with respect to $\nu$, and 
  \[
  \la_1+\la_2+\la_3 =1,\  \la_i \geq 0, i=1,2,3.
  \] 
  Set 
  $$F_s(x)=\mu_s([0,x]), \ \ f(x)= \frac{d\,\mu_{ac}}{d\,\nu}(x).$$ 
  
\vspace{0.4cm}
  
  \noindent {\bf Case 1:} The probability $\mu$ has countable number of atoms.
 
  \vspace{0.2cm}
  
  Since the total mass of $\mu_a$ is  equal to one,  there exists a countable infinite number of atoms with all different value of masses. Let  the masses of these atoms be ranked in descending order and the corresponding atoms are $x_1,x_2,\ldots$.  Clearly $\mu_s(\{x_i\})=\mu_{ac}(\{x_i\})=0$ for all $i\geq 1$. For any $m\geq 2$, by the continuity of probabilities, one can choose 
 small positive numbers $\ep_1,\ep_2, \ldots,\ep_m$  such that $x_i\pm \ep_i, 1\leq i \leq m$ are the continuity points of $\mu$, $(x_i-\ep_i,x_i+\ep_i]\subset [0,1], 1\leq i\leq m$ are disjoint, and  $$\mu((x_1-\ep_1, x_1+\ep_1])>\mu((x_2-\ep_2, x_2+\ep_2]) >\cdots>\mu((x_m-\ep_m, x_m+\ep_m]) .$$ The partition based on the points $\{x_i\pm \ep_i\: i=1,2,\ldots,m\}$ clearly gives a lower bound  $m-1$ for ${\cal I}(\cdot)$.  Since $m$ is arbitrary, the supremum taken over continuity points of $\mu$ gives the value of infinity which is the same as  ${\cal I}(\cdot)$. 
  
\vspace{0.2cm}  
  
\noindent {\bf Case 2:} The probability $\mu$ has at most finite number of atoms and $\nu(\{f(x)\neq 1\})>0$.  

\vspace{0.2cm}

Let $A =\{x\in [0,1]: f(x) <1\}, B=\{x\in [0,1]:f(x)>1\}$, and $C=\{x\in [0,1]: f(x)=1\}$. Then we have 
\[
\mu_{ac}(A)< \nu(A), \mu_{ac}(B)>\nu(B), \mu_{ac}(C)=\nu(C)\]
and
\[
\nu(A)-\mu_{ac}(A)=\mu_{ac}(B)-\nu(B)
\] 
The fact that $\nu\{C\}<1$ thus implies that $\nu(A)>0, \nu(B)>0.$ For any $m\geq 1$ we can find 
$0<s_1<\cdots<s_m<1, 0<t_1<\cdots<t_m<1$ such that 
\beq 
&&\{s_i\}_{1\leq i\leq m}\subset A,  \{t_i\}_{1\leq i\leq m}\subset B\\
&& \{s_i, t_i\}_{i\geq 1}\ \mbox{does not contain atoms of }\  \mu\\ 
&&  \mbox{when}\ \la_2 >0, F'_s(x)=0 \ \mbox {for}\ x=s_i \ \mbox{or}\ t_i, i\geq 1.
\eeq  
For any $i,j \geq 1$, we then have 
\beq
\lim_{\ep \ra 0}\frac{\mu((s_i-\ep, s_i+\ep])}{2\ep}&=& \la_3\lim_{\ep \ra 0}\frac{\mu_{ac}((s_i-\ep, s_i+\ep])}{2\ep}=\la_3 f(s_i)\\
&<& \la_3 f(t_j)=\la_3\lim_{\ep \ra 0}\frac{\mu_{ac}((t_j-\ep, t_j+\ep])}{2\ep}\\
&=&\lim_{\ep \ra 0}\frac{\mu((t_j-\ep, t_j+\ep])}{2\ep}.
\eeq

This makes it possible to choose $\ep_i>0$ such that  $s_i\pm \ep_i, t_j\pm \ep_j$ are all continuity points of $\mu$ and 
\[
\frac{\mu((s_i-\ep_i, s_i+\ep_i])}{\nu(s_i-\ep_i, s_i+\ep_i]} < \frac{\mu((t_j-\ep_j, t_j+\ep_j])}{\nu(t_j-\ep_j, t_j+\ep_j]}.\] 

This provides a lower bound of $m$ for $ {\cal I}(\mu)$. Since $m$ is arbitrary, we established \rf{fgz-5}  in this case.  
 
\vspace{0.2cm}  
  
\noindent {\bf Case 3:} The probability $\mu$ has at most finite number of atoms, $\la_2>0$ and  $\nu(\{f(x)\neq 1\})=0$.  

\vspace{0.2cm}  

It is clear that we have $\mu_{ac}=\nu$ in this case.  For any $m \geq 1$, the  singularity guarantees the existence of $0<s_1<\cdots<s_m<1, 0<t_1<\cdots<t_m<1$
 such that the derivative of $F_{s}(x)$ is zero for $x=t_i$ while the derivative at $s_i$ is either infinity or does not exist. Additionally we can choose $s_i,t_i$ so that none of them are atoms of $\mu_a$.   Let $\ep$ be small enough so that all  intervals 
 $(s_i-\ep,s_i+\ep]$ and $(t_i-\ep,t_i+\ep]$ $i=1, \ldots, m$ are disjoint.  Let ${\cal J}$  denote the partition of $[0,1]$ using $\{t_i\pm \ep,s_i\pm\ep: i=1,\ldots,m\}$.  
 One can then find a refined partition, using subsequence if necessary, $\tilde{\cal J}$ of ${\cal J}$, and positive numbers $\ep_0, \delta_0$ such that $s_i\pm \ep_0, t_i\pm \ep_0$ are continuity points of $\mu$ and  the value of $(2\ep_0)^{-1} \mu_{s}$ on each interval containing one of the $t_i'$s is less than $\delta_0$ while its value on each interval containing one of the $s_i'$s is greater than $\delta_0$.  In other words,  we can have for any $1\leq i,j \leq m$
 \[
 \frac{\mu((s_i-\ep_0, s_i+\ep_0])}{\nu((s_i-\ep_0, s_i+\ep_0])}\neq \frac{\mu((t_j-\ep_0, t_j+\ep_0])}{\nu((t_j-\ep_0, t_j+\ep_0])}. \]
 This implies that 
 \[
 \sup\{{\cal I}_n(\mu([0,t_1]),\mu((t_1,t_2]),\ldots, \mu((t_n,1]):
  0=t_0<t_1<\cdots<t_n<t_{n+1}=1, n \geq 1 \}\geq m. \]
The arbitrary selection of $m$ leads to  \rf{fgz-5} in this case.

\vspace{0.2cm}  
  
\noindent {\bf Case 4:} The probability $\mu$ has at most finite number of atoms, $\la_2=0$ and  $\nu(\{f(x)\neq 1\})=0$.  

\vspace{0.2cm}  

In this  case we have  $\mu=\la_1\mu_a +\la_3\nu$.  If $\la_1 =0$, then $\mu=\nu$ and ${\cal I}(\mu)$ is clearly zero.  Assume that $\la_1>0$ and the number of atoms is $r$. Let $F(x)=\mu([0,x])$. Since $r$ is finite,  any partition ${\cal J}$ of $[0,1]$ will have at most $r$ disjoint intervals covering these atoms. The maximum
 \[
 \sup\{{\cal I}_n(\mu([0,t_1]),\mu((t_1,t_2]),\ldots, \mu((t_n,1]):
 0=t_0<t_1<\cdots<t_n<t_{n+1}=1, n \geq 1\} \]
 is achieved at any partition with exactly $r$ disjoint intervals covering the $r$ atoms.

   \hfill $\Box$

   {\bf Proof of Theorem~\ref{fgz-t4}:} Let $C([0,1])$ be the space of all continuous function on $[0,1]$ equipped with the supremum norm, and  $\{g_j(x): j=1,2,... \}$ be a countable dense subset of $C([0,1])$. The set $\{g_j(x): j=1,2,... \}$ is clearly convergence determining
    on $M_1([0,1])$. Let $|g_j|=\sup_{x \in [0,1]}|g_j(x)|$ and $\{h_j(x)=\frac{g_j(x)}{|g_j|\vee 1}: j=1,...\}$
  is also convergence determining.

     For any $ \mu, \upsilon$ in $M_1([0,1])$, define
    \be\label{twodir20}
     d(\mu,\upsilon)=\sum_{j=1}^{\infty}\frac{1}{2^j}|\langle \mu, h_j\rangle-\langle\upsilon,h_j\rangle|.
    \ee
   Then $d$ is a metric generating the weak topology on  $M_1([0,])$.

    For any $\delta >0, \mu \in M_1([0,1])$, let
\[
B(\mu,\delta)=\{\upsilon \in M_1([0,1]): d(\upsilon,\mu)< \delta\}, \ \ \overline{B}(\mu,\delta)=\{\upsilon \in M_1([0,1]): d(\upsilon,\mu)\leq \delta\}.
\]

    Since $M_1([0,1])$ is compact, the family of the laws of $\Xi_{\alpha, 0, \nu}$ is exponentially tight. By theorem (P) in \cite{Pu91}, to prove the theorem it suffices to verify that

    \beqn\label{fgz-6}
 &&\lim_{\delta \ra
0}\liminf_{\alpha \ra 1 }\frac{1}{-\log(1-\alpha)}\log \mathbb{P}\{B(\mu,\delta)\} \\
&&= \lim_{\delta \ra
0}\limsup_{\alpha \ra 1}\frac{1}{-\log(1-\alpha)} \log \mathbb{P}\{\overline{B}(\mu,\delta)\} =- {\cal I}(\mu).\nn
\eeqn

Let $m$ be large enough so that
\be\label{addition1}
\{\upsilon \in M_1([0,1]):|\langle \mu, h_j \rangle -\langle \upsilon, h_j \rangle|< \delta/2:j=1,\cdots,m \} \subset B(\nu,\delta).
\ee

Consider $0=t_0<t_1< \cdots< t_n <t_{n+1}=1$ with $A_i=(t_{i-1},t_i], i=1, \ldots, n+1$ such that
\[
 \sup\{|h_j(x)-h_j(y)|: x,y \in A_i,i=1,\cdots,n; j=1,\cdots,m \}< \delta/8.
\]

Choosing $0< \delta_1 < \frac{\delta}{4n}$, and define
\[
 V_{t_1, \cdots, t_n}(\mu, \delta_1)= \{(y_1,...,y_n) \in \triangle_n: |y_i-\mu(A_i)|< \delta_1 , i=1,\cdots,n\}.
\]

For any $\upsilon$ in $M_1([0,1])$, let
\[
\Psi(\upsilon)=(\upsilon(A_1),...,\upsilon(A_{n+1})).
\]

If $\Psi(\upsilon)$ belongs to  $ V_{t_1, \cdots, t_n}(\mu, \delta_1)$, then for $j=1,...,m$
\beq
|\langle \upsilon, h_j \rangle -\langle \mu, h_j \rangle|&=&|\sum_{i=1}^{n+1}\int_{A_i}h_j(x)(\upsilon(dx)-\mu(dx))|\\
&<& \frac{\delta}{4} + n\delta_1  <  \delta/2,
\eeq
which implies that
\[
\Psi^{-1}( V_{t_1, \cdots, t_n}(\mu, \delta_1))\subset
\{\upsilon\in M_1([0,1]):|\langle \upsilon, h_j \rangle -\langle \mu, h_j \rangle|< \delta/2:j=1,\cdots,m \}.
\]
This combined with \rf{addition1} implies that
\[
 \Psi^{-1}(V_{t_1, \cdots, t_n}(\mu, \delta_1))\subset B(\mu,\delta).
\]

Since $V_{t_1, \cdots, t_n}(\mu, \delta_1)$ is open in $\triangle_{n}$, it follows from Lemma~\ref{fgz-l2} that
\beqn
 &&\lim_{\delta \ra
0}\liminf_{\alpha \ra 1}\frac{1}{-\log(1-\alpha)}\log \mathbb{P}\{B(\mu,\delta)\}\label{fgz-7}\\
&&\hspace{0.5cm}\geq
\lim_{\delta \ra
0}\liminf_{\alpha \ra 1}\frac{1}{-\log(1-\alpha)}\log \mathbb{P}\{\Psi^{-1}(V_{t_1, \cdots, t_n}(\mu, \delta_1))\}\nn\\
&&\hspace{0.5cm}= \lim_{\delta \ra
0}\liminf_{\alpha \ra 1}\frac{1}{-\log(1-\alpha)}\log \mathbb{P}\{(\Xi_{\alpha,0,\nu}(A_1),...,\Xi_{\alpha,0,\nu}(A_{n+1}))\in V_{t_1, \cdots, t_n}(\mu, \delta_1)\}\nn\\
&&\hspace{0.5cm}\geq -{\cal I}_{n+1}(\mu(A_1),...,\mu(A_{n+1}))\geq -{\cal I}(\mu).\nn
\eeqn

Next we assume that $t_1,...,t_n$  are continuity points of $\mu$. We denote the collection of all partitions from these points 
by ${\cal J}_{\mu}$. This implies that $\Psi(\upsilon)$ is continuous
 at $\mu$. Hence for any $\delta_2 >0$, one can choose $\delta >0$ small enough such that
\[
\overline{B}(\mu,\delta) \subset \Psi^{-1}(V_{t_1, \cdots, t_k}(\mu, \delta_2)).
\]

Let
\[
 \overline{V}_{t_1, \cdots, t_k}(\mu, \delta_2)=\{(y_1,...,y_n)\in \triangle_n: |y_i-\mu(A_i)|\leq \delta_2 , i=1,\cdots,n\}.
\]

Then we have

\beqn
&& \lim_{\delta\ra 0}\limsup_{\alpha \ra 1}\frac{1}{-\log(1-\alpha)} \log \mathbb{P}\{\overline{B}(\mu,\delta)\} \label{fgz-8}\\
&& \hspace{0.5cm}\leq  \limsup_{\alpha \ra 1}\frac{1}{-\log(1-\alpha)} \log \mathbb{P}\{(\Xi_{\alpha,0,\nu}(A_1),...,\Xi_{\alpha,0,\nu}(A_{n+1}))\in
\overline{V}_{t_1, \cdots, t_n}(\mu, \delta_2)\}.\nn
\eeqn

Letting $\delta_2$ go to zero and applying Lemma~\ref{fgz-l2} again, one gets

\[
\lim_{\delta\ra 0}\limsup_{\theta \ra \infty}\frac{1}{\theta} \log P\{\overline{B}(\mu,\delta)\} \leq -
{\cal I}_{n+1}(\mu(A_1),...,\mu(A_{n+1})).
\]

Finally, taking supremum over ${\cal J}_{\mu}$ and applying Lemma~\ref{fgz-l3}, one gets
\[
\lim_{\delta\ra 0}\limsup_{\alpha \ra 1}\frac{1}{-\log(1-\alpha)} \log \mathbb{P}\{\overline{B}(\mu,\delta)\} \leq -{\cal I}(\mu),
\]
which combined with \rf{fgz-7} leads to the theorem.

    \hfill $\Box$
  
   \section{Concluding Remarks}

The limiting procedure $\alpha$ going to zero arises naturally in the branching model considered in \cite{Athreya12}. This is a Galton-Watson branching process with offspring distribution $\{p_j: j\geq 0\}$ in the domain of attraction of a stable law of order $0<\alpha<1$ and $p_0=0$. For any $n \geq 1$, let $T_n$ be the coalescence  time  of any  two randomly selected individuals from the nth
generation. Then it is shown in \cite{Athreya12}  that
   \[
   \mathbb{P}\{n-T_n\leq k\} \ra \pi(k) \ \mbox{as}\ n \ra \infty
   \]
   where $\pi(k)$ is identified as the expectation of a random variable. It turns out that the random variable is just $\varphi_2({\bf P}(\alpha^k,0))$ and  $\pi(k)$ has the following more  explicit expression 
   \[
   \pi(k)=1-\alpha^k.
   \]
 In this context, $\varphi_2({\bf P}(\alpha^k,0))$ gives  the random probability distribution of the coalescence time and its asymptotic behaviour for large $k$ or equivalently $\alpha^k$ going to zero is described in Theorem~\ref{sec3-t1} and Theorem~\ref{ldp-t1}.  
     
 A comparison between  $\alpha$ converging to $1$ and $\theta$ converging to infinity reveals fundamental differences.  Under these limiting procedures, we have both ${\bf P}(\alpha,0)$ and  ${\bf P}(0,\theta)$ converge to  $(0,0,\ldots)$.   This can be seen from the distributions of $\varphi_2({\bf P}(\alpha,0))$ and $\varphi_2({\bf P}(0,\theta))$. 
  
 It is shown in \cite{Gri79a} and \cite{JKK02} that
  \[
  \sqrt{\theta/2} [\theta \varphi_2({\bf P}(0,\theta))-1] \Longrightarrow Z, \  \theta \ra \infty,
  \]
   where $Z$ is the standard normal random variable.  By Ewens sampling formula, we have
   \beq
   \mathbb{E}[\varphi_2({\bf P}(0,\theta))]&=& \frac{1}{\theta+1}\\
   \mathbb{E}[\varphi^2_2({\bf P}(0,\theta))]&=& \frac{3!+ \theta}{(\theta+1)(\theta+2)(\theta+3)} 
   \eeq
   and 
   \[
    \mathbb{E}[\varphi^3_2({\bf P}(0,\theta))]= \frac{1}{(\theta+1)_{(5)}}(5!+ 3\cdot 3!\theta+\theta^2).    
    \]
   
   The skewness of $\varphi_2({\bf P}(0,\theta))$  is given by
   \beq
   &&\frac{\mathbb{E}[\varphi_2^3({\bf P}(0,\theta))]- 3\mathbb{E}[\varphi_2({\bf P}(0,\theta))]\mathbb{E}[\varphi_2^2({\bf P}(0,\theta))]+2(\mathbb{E}[\varphi_2({\bf P}(0,\theta))])^3}{(\mathbb{E}[\varphi_2^2({\bf P}(0,\theta))]-(\mathbb{E}[\varphi_2({\bf P}(0,\theta))])^2)^{3/2}}\\
   && =\frac{O(\theta^{-5})}{O(\theta^{-4.5})} \ra 0, \ \   \theta\ra \infty
   \eeq
   which is consistent with the Gaussian limit.
   
   On the other hand, for $\varphi_2({\bf P}(\alpha,0))$ one has 
   \beq
   \mathbb{E}[\varphi_2({\bf P}(\alpha,0))] &=& 1-\alpha\\
    \mathbb{E}[\varphi_2^2({\bf P}(\alpha,0))] &= &\frac{(1-\alpha)(2-\alpha)(3-\alpha)+\alpha(1-\alpha)^2}{6},
   \eeq
   and 
   \beq
   Var(\varphi_2({\bf P}(\alpha,0)))&=& \frac{\alpha(1-\alpha)}{3}\\
   \mathbb{E}[\varphi_2^3({\bf P}(\alpha,0))]&=& \frac{1}{5!}[(1-\alpha)_{(5)}+ 3\alpha (1-\alpha)^2 (2-\alpha)(3-\alpha)+ \alpha^2(1-\alpha)^3].  
   \eeq
   
   This means that the skewness of $\varphi_2({\bf P}(\alpha,0))$ is of order $O((1-\alpha)/O((1-\alpha)^{3/2})$ which  goes to infinity as $\alpha$ converges to $1$. Thus the distribution of 
   $\varphi_2({\bf P}(\alpha,0))$ is skewed strongly to the right and a Gaussian  limit is unlikely.

   Another difference is reflected from the large deviation behaviour of the Pitman sampling formula.  For any $n \geq 1$, a partition {\boldmath$\eta$}  of $n$ with length $l$, the conditional Pitman sampling formula given ${\bf P}(\alpha,\theta)={\bf p}$ is 
    \[
   F_{\mbox{\boldmath$\eta$}}({\bf p})= C(n,\mbox{\boldmath$\eta$})\sum_{\mbox{distinct}\ i_1,\ldots,i_l }p_{i_1}^{\eta_1}\cdots p_{i_l}^{\eta_l}
   \]
   where 
   \[
   C(n,\mbox{\boldmath$\eta$})= \frac{n!}{\prod_{k=1}^l \eta_k!\prod_{j=1}^na_j(\mbox{\boldmath$\eta$})}.
   \]
   
  Assuming $\eta_i\geq 2$ for all $i$. Then $ F_{\mbox{\boldmath$\eta$}}({\bf p})$ is continuous on $\nabla_{\infty}$. By contraction principle, large deviation principles hold for the image laws of $PD(0,\theta)$ and $PD(\alpha,0)$ under $F_{\mbox{\boldmath$\eta$}}({\bf p})$ with respective speed $\theta$ and $-\log(1-\alpha)$.
  
  Integrating $F_{\mbox{\boldmath$\eta$}}({\bf p})$ with respect to $PD(\alpha, \theta)$ leads to the unconditional Pitman sampling formula. The large deviation speed is  shown in \cite{Feng07} to be 
   $\log \theta$ under $PD(0,\theta)$. In \cite{FengZhou15}, the large deviation  speed under $PD(\alpha,0)$ is shown to be  $-\log(1-\alpha)$. In other words, under $PD(0,\theta)$ the conditional and unconditional Pitman sampling formulae have different large deviation speeds due to averaging and finite sample size,  while under $PD(\alpha,0)$ the corresponding speeds are the same.  
    
   The large deviations  for $\Xi_{\alpha,0,\nu}$ provide more information on the microscopic transition structure at the critical  temperature for the REM. At the instant when the temperature starts to move below the critical value $T_c$, a portion of mass of the uniform measure $\nu$ may be lost and is replaced by an atomic portion with finite atoms. This represents the emerging of  finite number of energy valleys and the energy landscape of the system becomes a mixture of valleys and ``flat" regions. The emerging of energy valleys follow the order where the small number of energy valleys is more likely to occur than a large number of valleys.

      \end{document}